\documentclass[a4paper]{amsart} 

\usepackage[colorlinks,linkcolor=blue,citecolor=blue,urlcolor=blue]{hyperref}
\usepackage[english]{babel}
\usepackage{amsthm,amsfonts,amssymb,mathrsfs,amsmath}
\usepackage[shortlabels]{enumitem}
\usepackage{xcolor}
\usepackage{soul}

\usepackage{mathtools}						
\mathtoolsset{showonlyrefs,showmanualtags}	

\usepackage{todonotes}  

\newtheorem{Theorem}{Theorem}[section]
\newtheorem{Coro}[Theorem]{Corollary}
\newtheorem{Lemma}[Theorem]{Lemma}
\newtheorem{Prop}[Theorem]{Proposition}

\theoremstyle{definition}
\newtheorem{Def}[Theorem]{Definition}

\theoremstyle{remark}
\newtheorem{rk}[Theorem]{Remark}

\renewcommand{\aa}{\alpha}

\newcommand{\eps}{\varepsilon}

\newcommand{\R}{\mathbb{R}}

\newcommand{\N}{\mathbb{N}}
\newcommand{\C}{\mathbb{C}}

\numberwithin{equation}{section}

\title[On sub-Riemannian Weyl's theorems]  {On Weyl's type theorems and genericity of projective rigidity in sub-Riemannian Geometry}
\date{\today}

\thanks{This work was supported by a public grant as part of the
	Investissement d'avenir project, reference ANR-11-LABX-0056-LMH, LabEx
	LMH, in a joint call with Programme Gaspard Monge en Optimisation et
	Recherche Op\'{e}rationnelle, by the iCODE
	Institute project funded by the IDEX Paris-Saclay, ANR-11-IDEX-0003-02 and by the Grant ANR-15-CE40-0018 of the ANR. I.\ Zelenko was partly supported by NSF grant DMS-1406193 and Simons Foundation Collaboration Grant for Mathematicians 524213.}

\author{Fr\'ed\'eric Jean}
\address{Fr\'ed\'eric Jean \\
	UMA, ENSTA Paris, Institut Polytechnique de Paris, F-91120 Palaiseau, France}
\email{frederic.jean@ensta-paris.fr}
\urladdr{\url{http://uma.ensta-paris.fr/~fjean}}

\author{Sofya Maslovskaya}
\address{Sofya Maslovskaya \\
	INRIA Sophia Antipolis, team Biocore, 2004 Route des Lucioles, BP93
	06902 Sophia Antipolis Cedex, France}
\email{Sofya.Maslovskaya@inria.fr}

\author{Igor Zelenko}
\address{Igor Zelenko\\
	Department of Mathematics\\
	Texas A\&M University\\
	College Station\\
	Texas \ 77843\\
	USA}
\email{zelenko@math.tamu.edu}
\urladdr{\url{http://www.math.tamu.edu/~zelenko}}

\begin{document}
\begin{abstract}		
	H.~Weyl in 1921 (\cite{W}) demonstrated that for a connected manifold of dimension greater than $1$, if two Riemannian metrics are conformal and have the same geodesics up to a reparametrization, then one metric is a constant scaling of the other one. In the present paper we investigate the analogous property for sub-Riemannian metrics. In particular, we prove  that the analogous statement, called the \emph{Weyl projective rigidity}, holds either in real analytic category for all sub-Riemannian metrics on distributions with a specific property of their complex abnormal extremals, called \emph{minimal order}, or in smooth category for  all distributions such that all complex abnormal extremals of their nilpotent approximations are of minimal order.
	This also shows, in real analytic category,  the genericity  of distributions for which all sub-Riemannian metrics are Weyl projectively rigid and genericity of Weyl projectively rigid sub-Riemannian metrics on a given bracket generating distributions. Finally, this  allows us to get analogous genericity results for projective rigidity of sub-Riemannian metrics, i.e.\ when the only sub-Riemannian metric having the same sub-Riemannian geodesics , up to a reparametrization, with a given one, is a constant scaling of this given one. This is the improvement of our results on the genericity of weaker rigidity properties proved in recent paper \cite{jmz2018}.
\end{abstract}
\maketitle

\section{Statement of the problem and main results}

In Riemannian geometry, projectively (or geodesically) equivalent metrics are Riemannian metrics on the same manifold which have the same geodesics, up to reparameterization. The local classification of all pairs of projectively equivalent Riemannian metrics under natural regularity assumptions was maid  by Levi-Civita in 1896
(\cite{Levi-Civita1896}). This paper is devoted to the projective equivalence of more general class of metrics, the sub-Riemannian metrics and is a continuation of our recent work \cite{jmz2018}.

A sub-Riemannian manifold is a triple $(M,D,g)$, where $M$ is a smooth connected manifold, $D$ is a distribution on $M$ (i.e. a subbundle of $TM$) which is assumed to be bracket generating everywhere in the sequel without special mentioning, and $g$ is a Riemannian metric on $D$, and thus defines an Euclidean structure   on every fiber of $D$. We say that $g$ is a sub-Riemannian metric on $(M,D)$.  Consider the optimal control problem of minimizing the corresponding energy functional $E(\gamma) = \int g(\dot{\gamma},\dot{\gamma})\,dt$ on the space of absolutely continuous curve tangent to $D$. The geodesics  of the sub-Riemannian metric $g$ are projections of the Pontryagin extremals for this problem. Sub-Riemannian Pontryagin extremals and the corresponding geodesics  can be of two types, normal or abnormal.

The normal Pontryagin extremal  of the sub-Riemannian metric are integral curves of the Hamiltonian system for the corresponding Hamiltonian $h$, living on a \emph{ nonzero level} set of this Hamiltonian.   The \emph{Hamiltonian}  $h: T^*M \rightarrow \R$ of the sub-Riemannian metric $g$ is defined by
\begin{equation}
\label{Ham}
h(q,p) = \frac{1}{2} \|  p \|_q^2, \qquad q\in M, \ p \in T^*_qM,
\end{equation}
where $$
\| p \|_q = \max \left\{ \left< p, v\right> :  v \in D(q), \ g(q)(v,v) = 1 \right\}, \qquad p \in T^*_q M.
$$
The abnormal Pontryagin extremals live in the \emph{zero level set of} $h$, or, equivalent, on the annihilator $D^\perp$ of the distribution $D$, i.e.
	\begin{equation}
	\label{Dperp}
	D^\perp=\{(q,p)\in T^*M: q\in M, \ p\in T_q^*M, \ p|_{D(q)}=0\}.
	\end{equation}
Their description is more involved, as they are not the integral curve of the sub-Riemannian Hamiltonian $h$, and will be given in section \ref{gensec}.  Abnormal geodesics depend only on the distribution $D$, not on $g$, so they are automatically the same for all sub-Riemannian metrics on the same distribution.

Riemannian metrics appear as the particular case of sub-Riemannian ones, where $D=TM$. The classical Riemannian geodesics can be equivalently described as the normal geodesics coming from the corresponding Hamiltonian \eqref{Ham}. Riemannian metrics do not have abnormal geodesics. We thus extend the definition of projectively equivalence to sub-Riemannian metrics in the following way.

\begin{Def}
	Let $M$ be a manifold and $D$ be a bracket generating distribution on $M$. Two sub-Riemannian metrics  $g_1$ and $g_2$ on  $(M,D)$ are called \emph{projectively equivalent} at $q_0\in M$ if they have the same geodesics, up to a reparameterization, in a neighborhood of $q_0$.
\end{Def}

The trivial example of projectively equivalent metrics is the one of two constantly proportional metrics $g$ and $cg$, where $c>0$ is a real number. We thus say that these metrics are \emph{trivially} (projectively or affinely) equivalent.

\begin{Def}
	A sub-Riemannian metric $g$ on $(M,D)$ is said to be \emph{projectively rigid} if it admits no non-trivially projectively equivalent metric.
\end{Def}

It is still a widely open problem to classify all pairs of projectively equivalent  sub-Riemannian metrics.
A much easier task  is to study whether the projectively rigid sub-Riemannian metrics form a generic set in the space of all sub-Riemannian metrics on a connected manifolds. In studying this question one naturally arrives to the following weaker (intermediate) notion of rigidity.

\begin{Def}
	A sub-Riemannian metric $g$ is said to be \emph{conformally projectively rigid} if any metric projectively equivalent to $g$ is conformal to $g$.
\end{Def}
In our recent paper \cite{jmz2018} we proved the following genericity results for conformally projective rigidity:

\begin{Theorem}
	\label{th:generic_metric}
	Let $M$ be a smooth manifold and $D$ be a smooth distribution on $M$. A generic sub-Riemannian metric on $(M,D)$ is conformally projectively rigid.
\end{Theorem}

\begin{Theorem}
	\label{th:generic_distrib}
	Let $m$ and $n$ be two integers such that $2\leq m <n$, and assume $(m,n) \neq (4,6)$ and $m \neq n-1$ if $n$ is even.
	Then, given a smooth  $n$-dimensional manifold $M$ and  a generic smooth rank-$m$ distribution $D$ on $M$,
	any sub-Riemannian metric on $(M,D)$ is conformally projectively rigid.
\end{Theorem}

The latter theorem is based on the following result also proved in \cite{jmz2018}.

\begin{Theorem}
	\label{Cor:anyrank}
	If $D$ is a bracket generating distribution on a connected manifold $M$  such that the nilpotent approximation of it at every point of an open and dense subset of $M$ does not admit a product structure, then any sub-Riemannian metric on $D$ is conformally projectively rigid.
\end{Theorem}

In light of these results,  it is natural to ask whether conformally projective rigidity can be replaced by just projective rigidity in both of these theorems. In the Riemannian case  H.~Weyl in 1921 (\cite{W}) demonstrated that for $\dim\, M>1$, if two Riemannian metrics are conformal and have the same geodesics up to a reparametrization, then one metric is a constant scaling of the other one.\footnote{In fact this is a simple consequence of the Levi-Civita classification in \cite{Levi-Civita1896} which was written much earlier than \cite{W} but we prefer to relate it first to H.~Weyl as the great founder of both conformal and projective geometry.}

\begin{Def}
	A metric $g$ is said to be \emph{Weyl projectively rigid} if  any metric, which is simultaneously conformal to $g$ and projectively equivalent to $g$ is constantly proportional to $g$.
\end{Def}

So, in this terminology the Weyl theorem says that for $\dim\, M>1$ any Riemannian metric is Weyl projectively rigid. While in Riemannian case the proof of this result is rather trivial, it is not known yet whether the same statement is true for all sub-Riemannian metrics. The main problem that arises in trying to prove this statement for the general sub-Riemannian case is the presence of abnormal extremals.
The main objective of the present  paper is to study \emph{under what conditions are sub-Riemannian metrics Weyl projectively rigid.} Studying the solvability of the equations for projective equivalence, one inevitably arrives to the questions of divisibility of certain polynomials on the fibers of the cotangent bundle $T^*M$, so it is natural to complexify the picture by complexifying not only the fibers of $T^*M$ but the manifold $M$ itself which is possible, at least locally, under assumption that $M$ is real analytic. This leads naturally to the necessity to consider the notion of complex extremals and geodesics. Very roughly speaking, we show that a real analytic sub-Riemannian metric is Weyl projectively rigid either if the underlying distribution does not have too much complex abnormal geodesics through a point or does not have too much complex \emph{non-strictly} normal geodesics, i.e. complex normal geodesics which are simultaneously abnormal. Our condition are already enough to prove the Weyl projective rigidity for appropriate generic class of sub-Riemannian metrics in real analytic category (see Theorems \ref{genthm1} and  \ref{genthm2}) below) so that it is possible to replace conformally projective rigidity by projectively rigid in Theorems \ref{th:generic_metric} and \ref{th:generic_distrib}.

Now we will describe our results in more detail.
Assume that $(M,D, g)$ is a real analytic sub-Riemannian structure. Locally (i.e.\ in a neighborhood of any points of $M$) we can consider a complex manifold $\mathbb C M$, a \emph{complexification of $M$}, by extending the transition maps between charts, which are real analytic by definition, to analytic functions. We can extend locally the (real-analytic) distribution $D$ and sub-Riemannian metric $g$ to the (complex) analytic distribution $\mathbb CD$ and  a field of symmetric forms $g^\mathbb {C}$ on each fiber of this distribution.

We can also consider the (complex) cotangent bundle $T^*\C M$ of $\C M$ whose fibers are complex vector spaces. We can extend the sub-Riemannian Hamiltonian $h$ defined by \eqref{Ham} analytically to the complex Hamiltonian $h^\mathbb C$ on this bundle and consider the corresponding complex Hamiltonian vector field. The \emph{complex normal extremals} are by definition the integral curves of this vector field and the \emph{complex normal geodesics} are projections of these integral curves to $\mathbb \C M$.
\begin{rk}
Note that after the complexification the zero-level set $(h^{\mathbb C})^{-1}(0)$ of the complex sub-Riemannian Hamiltonian $h^\mathbb C$ is strictly larger that the annihilator $\mathbb C D^\perp$. The integral curves of the complexified Hamiltonian lying in $(h^{\mathbb C})^{-1}(0)\backslash \mathbb CD^\perp$ play the same role as  null geodesics in the pseudo-Riemannian geometry and, in particular, they are the same, up to reparameterization, for all sub-Riemannian metrics from the same conformal class. We will call them and their projections to $\mathbb C M$  the \emph {complex null normal extremals and geodesics}, respectively.
\end{rk}

Also, we can define  Jacobi curve and the corresponding osculating flag for every complex normal extremal. Further,  manipulating with  the annihilators $(\mathbb C D)^\perp$ of the complex distribution $\mathbb CD$ in the complex cotangent bundle of $\C M$, similarly to the standard real case, we can define complex abnormal, and consequently the strictly normal sub-Riemannian geodesics  (see section \ref{gensec}  for more detail).  

We also need the notion of a corank of a geodesics. From now on by dimensions we will mean complex dimensions.
Given a complex normal geodesic $\gamma$  of a sub-Riemannian metric $g$ we say that a complex Pontryagin normal extremal projected to $\gamma$ as parameterized curve is a  \emph{Pontryagin normal lift of $\gamma$}.  Given an abnormal geodesic $\gamma$ of a distribution $D$  we say that a complex Pontryagin abnormal extremal projected to $\gamma$  is a  \emph{Pontryagin abnormal lift of $\gamma$}. The projection of this lift to the projectivized cotangent bundle $\mathbb P T^*M$ will be called the \emph{projectivized Pontryagin abnormal lift of $\gamma$}.  The
corank of a
complex normal  geodesics $\gamma$  of a sub-Riemannian metric $g$ is by definition the dimension of the affine space of
the
normal Pontryagin lifts of $\gamma$.
The
corank of a normal geodesic is a nonnegative integer.

The corank of a complex abnormal geodesic $\gamma$ of a distribution $D$ is by definition the dimension of the (vector) space of all its abnormal lifts.
The corank of an abnormal geodesic is a positive integer.

\begin{rk}
	\label{abnrem1}
	Note that if $\bigl(\gamma(t), p_1(t)\bigr)$ and $\bigl(\gamma(t), p_2(t)\bigr)$ are two distinct normal lifts of a normal geodesics $\gamma$ (which are either both null or both non-null), then   $\bigl(\gamma(t), p_2(t)-p_1(t)\bigr)$ is  an abnormal lift of $\gamma$. Similarly, if  $\bigl(\gamma(t), p_1(t)\bigr)$ is an  abnormal lift of a geodesic $\gamma$ and $\bigl(\gamma(t), p_2(t)\bigr)$  is a normal lift of $\gamma$, then $\bigl(\gamma(t), p_1(t)+p_2(t)\bigr)$  is a normal lift of $\gamma$.
\end{rk}

From the previous remark it follows that  a normal geodesic $\gamma$  is simultaneously an abnormal geodesic if and only if its corank is greater than $0$ and in  this case the corank of $\gamma$  as normal geodesics is equal to  the corank  of $\gamma$  as abnormal geodesics.

Further, if given positive integers $\kappa$ and $s$, there exists a nonzero number $c\in\mathbb C$ and  a $(s+ \kappa +1)$-dimensional submanifold of  the $c$-level set $(h^{\mathbb C})^{-1}(c)$ of the complexified sub-Riemannian Hamiltonian $h^{\mathbb C}$ which is  foliated by complex normal extremals of corank $\kappa$, then we say that the projections of these extremals to $M$ form an \emph{$s$-parametric family of complex
	normal geodesics of corank $\kappa$}.

Also note that the nilpotent approximation (even of a smooth but not real analytic sub-Riemannian structure) is always real analytic as it is a left-invariant structure on a nilpotent Lie group. The following two theorems are our main results on Weyl projective rigidity in terms of normal geodesics:

\begin{Theorem}
	\label{Weylcor2} Assume that $(M,D,g)$ is a smooth sub-Riemannian manifold such that its nilpotent approximation at every point  of an open and dense subset of $M$ satisfies the following property: for every positive $\kappa \in \N$, there is no $(n-2-\kappa)$-parametric family of corank $\kappa$ 
	 non-strictly normal complex geodesics through a point.
	Then the metric $g$ is Weyl projectively rigid.
\end{Theorem}

\begin{Theorem}
	\label{Weylcor3} 	
	Assume that $(M,D,g)$ is a real analytic sub-Riemannian manifold such that there is no open set $U$ in $\mathbb CM$ with the following property: for some positive integer $\kappa \leq n-2$ through any point $q\in U$  there is an   $(n-2-\kappa)$-parametric family of corank $\kappa$ non-strictly normal complex geodesics. Then the metric $g$ is Weyl projectively rigid.
\end{Theorem}

These two theorems will be proved in section \ref{famsec}, based on more general but technical Theorem \ref{mainthm}, proved in section \ref{mainthmsec}.

	The conclusion of Theorems \ref{Weylcor2} and \ref{Weylcor3} for a sub-Riemannian structure holds in particular when either the sub-Riemannian structure is smooth, and its nilpotent approximation do not have complex non-strictly normal geodesics or the sub-Riemannian structure is real analytic and  does not have complex non-strictly normal geodesics. The latter holds generically. It follows from the complex analog of \cite[Proposition 2.22]{cjt08}, which has literally  the same proof. To summarize, we have the following result on genericity of the Weyl rigidity:


\begin{Coro}
	\label{genthm1}	
	Let $M$ be a real analytic manifold and $D$ be a real analytic distribution on $M$ of rank greater than $1$. A generic real analytic sub-Riemannian metric on $(M,D)$ is Weyl  rigid.
\end{Coro}

Now we formulate our main results on Weyl projective rigidity  in terms of complex abnormal extremals. We will use the notion of an abnormal extremal of minimal order, introduced in \cite{chitour}, see Definition \ref{minorddef} below.  The condition of minimal order implies in particular that on a set of full measure such abnormal extremal is tangent to a prescribed line or equivalently, on a set of full measure the germ of the extremal at any point of this set is uniquely determined by this point.
The following two theorems will be proved in section \ref{gensec}.

	


\begin{Theorem}
	\label{Weylcor2a} Assume that $D$ is a smooth distribution on a connected manifold $M$ such that its nilpotent approximation at every point   of an open and dense subset of $M$ satisfies the following properties:
		every complex abnormal extremal of the nilpotent approximation  is of minimal order.
	Then any smooth sub-Riemannian  metric on $D$ is Weyl projectively rigid.
\end{Theorem}

\begin{Theorem}
	\label{Weylcor3a} 	
	Assume that $D$ is a real analytic distribution on a connected manifold  such that every complex abnormal extremal of $D$  is of minimal order.
	Then any real analytic sub-Riemannian  metric on $D$ is Weyl projectively rigid.

\end{Theorem}

The direct consequence of Theorems \ref{Cor:anyrank} and \ref{Weylcor2a} is the following.

\begin{Coro}
	\label{Cor:anyrankrigid}
	If $D$ is a smooth bracket generating distribution on a connected manifold $M$ such that the nilpotent approximation $\hat D$  of it at every point  of an open and dense subset of $M$ does not admit a product structure and   every complex abnormal extremal of $\hat D$   is of minimal order,
	 then any smooth sub-Riemannian metric on $D$ is projectively rigid.
\end{Coro}

Based on the last corollary we can easily find many new classes of distributions on connected manifolds for which all sub-Riemannian metrics on them are projectively rigid (before this statement was known for contact distributions only (\cite{Z})). For example, this will be true for the following distributions, for which it is easy to see that all possible nilpotent approximations satisfy conditions of Corollary  \ref{Cor:anyrankrigid}:

 \begin{enumerate}
 	\item
\emph{Engel distributions, i.e.\ rank $2$ distributions on $4$-dimensional  manifolds with the small growth vector $(2,3,4)$};
\item \emph{Rank $2$ distributions on $5$-dimensional  manifolds with the small growth vectors $(2,3,5)$;
\item rank $3$ distributions on $5$-dimensional and $6$ dimensional  manifolds with the small growth vectors $(3,5)$ and $(3,6)$, respectively};
\item \emph{Rank $2$ distributions on $6$ dimensional  manifolds with the small growth vectors $(2,3,5,6)$};
\item \emph{Rank $2$ distributions on $7$ dimensional  manifolds with the small growth vectors $(2,3,5,7)$ or $(2,3,5,6,7)$}.
\end{enumerate}

\begin{rk}
Conditions of minimal order in Theorems \ref{Weylcor2a} and \ref{Weylcor3a} can be replaced by much weaker but much more technically formulated condition, see subsection \ref{weaksec}.
\end{rk}


Further, the main result of \cite[Theorem 2.4]{chitour}  states that all (real) abnormal extremals of a generic smooth rank $m$ distribution are of minimal order and corank $1$. This result can be literally extended to the complex abnormal exttremals  of real analytic manifolds), because the genericity condition in \cite[Theorem 2.4]{chitour}  is given by the complement of algebraic conditions with respect to the fibers of $T^*M$ so that  the complexification can be done. So, the following theorem holds.

\begin{Theorem}
	\label{cjtthm}	
	All complex abnormal extremals of a generic rank $m$ real analytic distribution distribution are of minimal order and corank $1$.
\end{Theorem}

Combining this theorem with Theorem  \ref{Weylcor2a} we get one more  genericity results on the Weyl rigidity.

\begin{Coro}
	\label{genthm2}	
	Let $m$ and $n$ be two integers such that $2\leq m <n$.	On  a generic real analytic rank $m$ distribution $D$ on a connected  $n$-dimensional real analytic manifold $M$   any sub-Riemannian metric is Weyl projectively rigid.
\end{Coro}

Finally, as immediate consequences of Theorems \ref{th:generic_metric} and \ref{genthm1}  and Theorems \ref{th:generic_distrib} and \ref{genthm2}, respectively,  we get the following two genericity results for projective rigidity, improving the main results of \cite{jmz2018}.

\begin{Coro}
	\label{th:generic_metric
		_proj}
	Let $M$ be a real analytic  manifold and $D$ be a distribution on $M$. A generic real analytic  sub-Riemannian metric on $(M,D)$ is projectively rigid.
\end{Coro}

\begin{Coro}
	\label{th:generic_distrib_proj}
	Let $m$ and $n$ be two integers such that $2\leq m <n$, and assume $(m,n) \neq (4,6)$ and $m \neq n-1$ if $n$ is even.
	Then, given a smooth  $n$-dimensional manifold $M$ and  a generic smooth rank $m$ distribution $D$ on $M$,
	any sub-Riemannian metric on $(M,D)$ is projectively rigid.
\end{Coro}

\section{The fundamental algebraic system in the conformal case}

\subsection {Equations for orbital diffeomorphisms in local coordinates} In this subsection, following  \cite{jmz2018}, we introduce orbital diffeomorphisms between extremal flows and explain their relation to the projective equivalence,  then  deduce the equations for orbital diffeomorphisms in local metric for projective equivalent  and conformal sub-Riemannian metric. All formulas here can be directly derived from the corresponding formulas in  \cite {Z} and \cite{jmz2018}, where the general case (of not necessary conformal but projective equivalent sub-Riemannian metrics) is considered. To make the presentation self-contained we derive all formulas here in the particular conformal case.
	
Let $M$ be a manifold and $D$ be a bracket generating distribution on $M$. We consider two sub-Riemannian metrics on  $(M,D)$ that are both conformal and projectively equivalent. Let us denote these metrics by
 $g$ and $\alpha^2 g$, where $\alpha : M \to \R$ is a never vanishing smooth function. Let $h_1$ and $h_2$ be the sub-Riemannian Hamiltonians
 of $g$ and  $\alpha^2 g$, respectively. Obviously
 \begin{equation}
 \label{hh}
 h_2=\frac{1}{\alpha^2} h_1.
 \end{equation}
Denote  $H_1=h_1^{-1}(1/2)$ and $H_2=h_2^{-1}(1/2)$ the respective $\cfrac{1}{2}$-level sets of these Hamiltonians. Also let $\pi: T^*M\rightarrow M$ be the canonical projection.

One says that $\vec h_1$ and  $\vec h_2$ are \emph{orbitally diffeomorphic}
on an open subset $V_1$ of $H_1$ if there exists an open subset $V_2$ of $H_2$ and a diffeomorphism $\Phi: V_1 \rightarrow V_2$ such that $\Phi$ is fiber-preserving, i.e.\  $\pi (\Phi(\lambda)) = \pi (\lambda)$, and $\Phi$ sends the integral curves of $\vec h_1$ to the reparameterized integral curves of $\vec h_2$, i.e., there exists a smooth function $s=s(\lambda,t)$ with $s(\lambda,0)=0$ such that $\Phi\bigl(e^{t\vec h_1}\lambda\bigr) =e^{s\vec h_2} \bigl(\Phi(\lambda)\bigr)$ for all $\lambda \in V_1$ and $t\in \mathbb R$ for which $e^{t\vec h_1}\lambda$ is well defined. Equivalently, there exists a smooth function $c(\lambda)$ such that
\begin{equation}
\label{orbeq}
d\Phi \circ \vec{h}_1(\lambda) = c(\lambda) \vec{h}_2(\Phi(\lambda)).
\end{equation}
  The map $\Phi$  can be extended as a mapping $\bar \Phi$ from $T^*M \setminus h_1^{-1}(0)$ to itself by rescaling, i.e.,
$$
\bar \Phi (\lambda) = \sqrt{2 h_1(\lambda)} \Phi \left(\frac{\lambda}{\sqrt{2 h_1(\lambda)}}\right).
$$
The resulting map  is called an \emph{orbital diffeomorphism} between the extremal flows of $g$ and  $\alpha^2g$. In  the considered case,  from \eqref{hh} and the fact that $\Phi$ is fiber-preserving it follows immediately that the function $c(\lambda)$ in \eqref{orbeq}  coincides with the function $\aa\circ \pi (\lambda))$, i.e.\ we have
\begin{equation}
\label{orbeqc}
d\Phi \circ \vec{h}_1(\lambda) = \aa\circ \pi (\lambda) \vec{h}_2(\Phi(\lambda)).
\end{equation}

In \cite{jmz2018} we established the relationship between projective equivalence of sub-Riemannian  and orbital equivalence of the corresponding sub-Riemannian Hamiltonians. In particular,
 in Proposition 3.4 there we proved that  there exists a local orbital diffeomorphisms $\Phi$ between the Hamiltonian vector fields associated with $g$ and $\alpha^2 g$ near generic\footnote{In fact in the original formulation in \cite{jmz2018} we used the term ample instead of generic,  see [Definition 2.9] there, but we do not really need this technicalities here.} point of $T^*M$.

Now we will work in coordinates on fibers of $T^*M$ induced by an appropriate local moving frame on $M$. Fix a point $q_0 \in M$ and choose a frame $\{ X_1, \dots, X_n \}$ of $TM$ adapted to $D$ at $q_0$ such that $X_1, \dots, X_m$ is a $g$-orthonormal frame of $D$. At any point $q$ in a neighborhood $U$ of $q_0$, the basis $X_1(q), \dots, X_n(q)$ of $T_q M$ induces coordinates $(u_1,\dots,u_n)$  on $T^*_q M$ defined as $u_i(q,p) = \langle p, X_i(q) \rangle$. These coordinates in turn induce a basis $\partial_{u_1}, \dots, \partial_{u_n}$ of $T_{\lambda}(T^*_qM)$ for any $\lambda \in \pi^{-1}(q)$. For $i = 1, \dots, n$, we define the lift $Y_i$ of $X_i$ as the (local) vector field on $T^*M$ such that $\pi_* Y_i = X_i$ and $du_j (Y_i) = 0 \ \ \forall 1 \leq j \leq n$. In this way the local frame $\{ X_1, \dots, X_n\}$ on $M$ induces  the local frame
\begin{equation}	
\label{frame}
	\{ Y_1, \dots, Y_n, \partial_{u_1}, \dots, \partial_{u_n}\}
\end{equation}
on  $T^*M$. By a standard calculation, we obtain $h_1 = \frac{1}{2}\sum_{i=1}^{m} u_i^2$ and
	\begin{equation}
	\label{eq:vech}
	 \qquad \vec{h}_1=\sum_{i=1}^{m} u_i Y_i + \sum_{i=1}^{m} \sum_{j,k=1}^{n} c^{k}_{ij} u_i u_k \partial_{u_j},
	\end{equation}
where $c^k_{ij}$, $i,j,k \in \{1,\dots,n\}$ are the  \emph{structure functions} of the frame  $\{ X_1, \dots, X_n \}$, defined near $q_0$ by
$$
[X_i,X_j] = \sum_{k=1}^n c^k_{ij} X_k.
$$

Further, from \eqref{hh} it follows that $\vec{h}_2=\cfrac{1}{\alpha^2} \vec {h}_1+h_1\overrightarrow{\left(\frac{1}{\alpha^2}\right)}$ and so
\begin{equation}
\label{eq:h2}
\vec{h}_2=\frac{1}{\alpha^2} \Bigl(\sum_{i=1}^{m} u_iY_i + \sum_{i=1}^{m} \sum_{j,k=1}^{n} c^{k}_{ij} u_i u_k \partial_{u_j}\Bigr) - \sum_{j=1}^{n} \frac{1}{\alpha} X_j(\frac{1}{\alpha}) \Bigl(\sum_{i=1}^{m}  u_i^2\Bigr) \partial_{u_j}.
\end{equation}

Finally, let us denote by $\Phi_i(\lambda) , \ i= 1, \dots, n,$ the $u_i$-component  of $\Phi(\lambda)$ on the fiber, i.e.\ $\Phi_i(\lambda)= u_i\circ \Phi(\lambda)$.
From \eqref{hh} it follows that  \begin{equation}\label{firstm}
\Phi_k = \alpha u_k \qquad \hbox{for } k = 1, \dots, m.
\end{equation}


Substituting this into \eqref{orbeqc} we get the following:\footnote {From now on to simplify the notation in all relation involving functions on open sets of $T^*M$ $\alpha$ actually will mean $\alpha\circ \pi$.}

\begin{Lemma}
The map $\Phi$ is an orbital diffeomorphism between extremal flows of sub-Riemannian metrics $g$ and $\alpha^2 g$ if and only if the components $\Phi_{m+1},\ldots, \Phi_n$ satisfy the following system of equations:   	
\begin{align}
& \sum_{k= m+1}^{n} q_{jk} ( \Phi_{k} - \alpha u_k) = \sum_{i=1}^m (X_i(\alpha) u_j - X_j(\alpha) u_i)u_i,\quad j=1,\ldots, m
\label{eq:3.6}\\
& \vec{h}_1(\Phi_k - \alpha u_k) = \sum_{l= m+1}^{n} q_{kl} ( \Phi_{l} - \alpha u_l) + \label{eq:3.7} \\
& \ \qquad \qquad   \qquad \qquad + \sum_{i=1}^m (X_k(\alpha) u_i - X_i(\alpha) u_k) u_i, \quad k = m+1, \dots , n,
\end{align}
where $q_{jk} = \sum_{i=1}^m c_{ij}^k u_i$.
\end{Lemma}

Equations \eqref{eq:3.6}-\eqref{eq:3.7} are obtained by straightforward calculations in the moving frame \eqref{frame} after plugging equations \eqref{eq:vech}, \eqref{eq:h2}, and \eqref{firstm} into \eqref{orbeqc}. Equations \eqref{eq:3.6} are obtained by comparison of the components of $\partial_{u_j}$ of both sides of \eqref{orbeqc} with $j=1,\ldots, m$, while equations \eqref{eq:3.7} are obtained by comparison of the components of $\partial_{u_k}$ of both sides of \eqref{orbeqc} with $k=1,\ldots, m+1$.

\subsection{Fundamental algebraic system}  Now following  \cite{jmz2018} again we replace the system \eqref{eq:3.6}-\eqref{eq:3.7} that contains derivatives of the unknown functions $\Phi_k$, $k=m+1, \ldots, n$ by the (infinite) linear algebraic i.e. without derivatives) system for that unknown function that we call the \emph{fundamental algebraic system}. The process of obtaining the latter can be seen in a sense as the infinite prolongation of the subsystem given by \eqref{eq:3.6} using in each step of the prolongation  the equations from  \eqref{eq:3.7}.

In more details, in the first step one differentiate each of $m$ equations from \eqref{eq:3.6} in the direction of $h_1$  and replace each  $\vec h_1(\Phi_k-\alpha u_k)$ in the resulting  expression by
the right-hand side of \eqref{eq:3.7}. In this way we get new $m$ equations which are linear in $\Phi_k-\alpha u_k$. In the next step we differentiate these new $m$ equations  in the direction of $h_1$ and replace each  $\vec h_1(\Phi_k-\alpha u_k)$ in the resulting  expression by
the right-hand side of \eqref{eq:3.7} to obtain new $m$ equations which are linear in $\Phi_k-\alpha u_k$.	The fundamental algebraic system is obtained by repeating  this process infinitely many times. Setting $\widetilde{u}=(u_{m+1}, \dots, u_n)$ and $\widetilde{\Phi}=(\Phi_{m+1}, \dots, \Phi_n)$, the fundamental algebraic system \cite[(3.8)]{jmz2018} writes as
	 \begin{equation}
	 \label{eq:3.8}
	A (\widetilde{\Phi} - \alpha \widetilde{u}) = d,
	 \end{equation}
where $A$ is the matrix defined recursively in \cite[(3.10)]{jmz2018} and $d$ is a column vector with an infinite number of rows which can be decomposed in layers of $m$ rows as
\begin{equation}
\label{eq:3.9}
		d= \left(
		 \begin{array}{c}
		d^1 \\
		d^2 \\
		 \vdots \\
		d^{s} \\
		 \vdots \\
		 \end{array}
		 \right),
\end{equation}
where the coefficients  $d^{s}_j$, $1 \leq j \leq m$, of the vector $d^s \in \R^m$ are defined by
\begin{equation}
		 \label{eq:3.11}
		 \left\{
\begin{aligned}
		& d^{1}_{j} = \sum_{i=1}^m (X_i(\alpha) u_j - X_j(\alpha) u_i)u_i, \\
		& d^{s+1}_{j} = \ \vec{h}_1(d^s_{j}) + \sum_{k = m+1}^{n} a^s_{j,k} \sum_{i=1}^{m} u_i \left( X_i(\alpha) u_k - X_k(\alpha) u_i \right).  \\
		 \end{aligned}
		 \right.
\end{equation}
Note that by \cite[Proposition 3.11]{jmz2018} the matrix $A(u)$ is injective at a generic $u$.

\subsection{Sufficient conditions for Weyl rigidity in terms of solutions of the fundamental algebraic system}
The fundamental algebraic system \eqref{eq:3.8} implies that the coordinates of $\Phi$ are rational functions on the fibers. Proving that $g$ and $\alpha^2 g$ are proportional actually amounts to prove that these coordinates are polynomial, as stated below.

\begin{Prop}
\label{th:poly}
If there exists a local orbital diffeomorphism $\Phi$ which is polynomial on the fibers, then $g$ and $\alpha^2 g$ are locally constantly proportional, i.e., $\alpha$ is constant.
\end{Prop}

Before giving the proof of this result, we need to study the consequence of the fundamental algebraic system on the nilpotent approximation.

Fix a regular point $q_0$ and denote by $(\hat{M}, \hat{D})$ the nilpotent approximation of $(M, D)$ at $q_0$. We argue as in the proof of \cite[Theorem 7.1]{jmz2018}, with the same notations. In particular $\{\hat{X}_1, \dots, \hat{X}_n\}$ is a frame of $T \hat M$ adapted to $\hat D$ such that $\hat{X}_1, \dots, \hat{X}_m$ is $\hat{g}$-orthonormal and $\hat A$ is the matrix of \cite[Proposition 3.10]{jmz2018} constructed by using $\{\hat{X}_1, \dots, \hat{X}_n\}$ as a frame.

\begin{Lemma}
\label{le:APsi=d}
There exists one, and only one, solution $\Psi=(\Psi_{m+1}, \dots, \Psi_n)$ s of
\begin{equation}
\label{eq:fundamental_nilp}
\hat{A} \Psi=\hat{d},
\end{equation}
where, for any $s \in \N$ and $1 \leq j \leq m$, $\hat{d}^s_j$ is defined by
\begin{equation}
\label{eq:3.11'}
\left\{
\begin{aligned}
		& \hat{d}^{1}_{j} = \sum_{i=1}^m (X_i(\alpha)(q_0) u_j - X_j(\alpha)(q_0) u_i)u_i, \\
		& \hat{d}^{s+1}_{j} = \ \vec{\hat{h}}_1(\hat{d}^s_{j}) + \sum_{k = m+1}^{n} \hat{a}^s_{j,k} \sum_{i=1}^{m} u_i  X_i(\alpha)(q_0) u_k .  \\
\end{aligned}
\right.
\end{equation}
\end{Lemma}

\begin{proof}
An easy induction argument based on equations \eqref{eq:3.11} shows the following result, similar to \cite[Lemma 7.4]{jmz2018}: for any $s \in \N$ and $1 \leq j \leq m$, there hold:
\begin{itemize}
\item for every $q \in M$ near $q_0$,  $d^s_j$ is a polynomial in $u_1,\dots,u_n$ of weighted degree
$$
\mathrm{deg}_w (d^s_j) \leq 2s ;
$$

\item the homogeneous term of highest weighted degree in $d^s_j(q_0)$ is $\hat{d}^s_j$.
\end{itemize}
It results from \eqref{eq:3.8} that  $\begin{pmatrix} A & d \end{pmatrix}$ is not of full rank, thus also the matrix $\begin{pmatrix} \hat{A} & \hat{d} \end{pmatrix}$ is not of full rank. Since $\hat A$ is of full rank at a generic $u$  by \cite[Proposition 3.11]{jmz2018}, there exists a unique element in $\ker \begin{pmatrix} \hat{A} & \hat{d} \end{pmatrix}$ of the form
$(\Psi, -1)$, which ends the proof.
\end{proof}

Using all equations above it is easy to show that $\Psi$ has the following properties.
\begin{enumerate}[(i)]
  \item \label{list:psi_k} Each $\Psi_k$, $k=m+1, \dots, n$, is a rational function which is:
  \begin{itemize}
    \item homogeneous of degree $1$ w.r.t.\ the usual degree;

    \item $w$-homogeneous with $\mathrm{deg}_w(\Psi_k)=w_k-1$.
  \end{itemize}

\item For $j=1, \dots, m$, we have
\begin{equation}
\label{eq:3.6'}
\sum_{\{k \, : \, w_k=2\}} \sum_{i=1}^m \hat{c}_{ij}^k u_i \Psi_{k}  = \sum_{i=1}^m (\alpha^i u_j - \alpha^j u_i)u_i,
\end{equation}
where $\alpha^j=X_i(\alpha)(q_0)$.

\item For $k=m+1, \dots, n$, we have
  \begin{equation}
		 \label{eq:3.7'}
\vec{\hat{h}} (\Psi_k )= \sum_{\{l \, : \, w_l=w_k+1\}} \sum_{i=1}^m \hat{c}_{ik}^l u_i \Psi_{l} - \sum_{i=1}^m \alpha^i u_k u_i.
\end{equation}

\end{enumerate}

\begin{Lemma}
\label{prop:poly}
Assume that the map $\Psi$ given in Lemma~\ref{le:APsi=d} is polynomial. Then
\begin{equation}
X_1(\alpha)(q_0) = \cdots = X_m(\alpha)(q_0) = 0.
\end{equation}
\end{Lemma}

\begin{proof}
By hypothesis, every $\Psi_k$, $k=m+1, \dots, n$, is a polynomial. Moreover, by Property~\ref{list:psi_k} above, $\Psi_k$ is a linear function of $u$ and depends only on the coordinates $u_l$ of weight $w_l=w_k-1$. To simplify the notations, we use the following convention: given a positive integer $s$, an index $k_s$ denotes an index of weight $w_{k_s}=s$ and $\sum_{k_s}$ denotes  $\sum_{\{k_s \, : \, w_{k_s}=s\}}$. With this notation we have, for every $k_s$,
\begin{equation}
\label{eq:psi_eps}
\Psi_{k_s} = \sum_{k_{s-1}} \eps_{k_s k_{s-1}} u_{k_{s-1}},
\end{equation}
where the coefficients $\eps_{k_s k_{s-1}}$ are real numbers. Taking the derivative along $\vec{\hat{h}}$ we obtain
\begin{equation}
\label{eq:hpsi1}
\vec{\hat{h}}(\Psi_{k_s}) = \sum_{i=1}^m \sum_{k_{s-1}, l_s} \eps_{k_s k_{s-1}} \hat{c}_{ik_{s-1}}^{l_s} u_i u_{l_s}.
\end{equation}
On the other hand, plugging \eqref{eq:psi_eps} into \eqref{eq:3.7'}, we get
\begin{equation}
\label{eq:hpsi2}
\vec{\hat{h}}(\Psi_{k_s}) = \sum_{i=1}^m \sum_{l_s, l_{s+1}} \hat{c}_{ik_{s}}^{l_{s+1}}\eps_{l_{s+1} l_{s}} u_i u_{l_s} - \sum_{i=1}^m \alpha^i u_k u_i.
\end{equation}
Fix an index $i \in \{1,\dots,m\}$. By identifying the coefficients of the monomial $u_i u_{k_s}$ in \eqref{eq:hpsi1} and \eqref{eq:hpsi2}, we obtain the following equality,
\begin{equation}
\label{eq:ks}
\sum_{k_{s-1}} \eps_{k_s k_{s-1}} \hat{c}_{ik_{s-1}}^{k_s} =  \sum_{k_{s+1}} \hat{c}_{ik_{s}}^{k_{s+1}}\eps_{k_{s+1} k_{s}}  - \alpha^i,
\end{equation}
and, after a summation on the $n_s-n_{s-1}$ indices $k_s$,
\begin{equation}
\label{eq:Ki}
\sum_{k_{s-1}, k_s} \hat{c}_{ik_{s-1}}^{k_s} \eps_{k_s k_{s-1}} =  \sum_{k_s,k_{s+1}} \hat{c}_{ik_{s}}^{k_{s+1}}\eps_{k_{s+1} k_{s}}  - (n_s-n_{s-1}) \alpha^i.
\end{equation}

Set $K_i(s) = \sum_{k_s,k_{s+1}} \hat{c}_{ik_{s}}^{k_{s+1}}\eps_{k_{s+1} k_{s}}$. Then the above equality writes as
\begin{equation}
\label{eq:Ki_recs}
K_i(s-1) =  K_i(s)  - (n_s-n_{s-1}) \alpha^i \qquad \hbox{for } s > 1.
\end{equation}
Note that $K_i(r)=0$ since $r$ is the nilpotency step. Hence,
\begin{equation}
\label{eq:Ki_rec}
K_i(1) =    - (n_r-n_{1}) \alpha^i = - (n-m) \alpha^i.
\end{equation}
Now, by plugging \eqref{eq:psi_eps} in \eqref{eq:3.6'}, we have, for $j=1,\dots,m$:
\begin{equation}
\label{eq:3.6_eps}
\sum_{k_2,k_1} \sum_{i=1}^m \hat{c}_{ij}^{k_2} \eps_{k_2 k_{1}} u_i u_{k_1}  = \sum_{i=1}^m (\alpha^i u_j - \alpha^j u_i)u_i.
\end{equation}
Given an index $i \in \{1,\dots,m\}$, the identification of coefficient of $u_i u_j$ in this equality gives
\begin{equation}
\label{eq:3.6_ij}
\sum_{k_2} \hat{c}_{ij}^{k_2} \eps_{k_2 j}   = (1 - \delta_{ij}) \alpha^i,
\end{equation}
and by summation on the indices $j=k_1$, we obtain
\begin{equation}
\label{eq:Ki_1}
K_i(1)  = (m-1) \alpha^i.
\end{equation}
This equation and \eqref{eq:Ki_rec} imply $\alpha^i=0$, which ends the proof.
\end{proof}


\begin{proof}[Proof of Proposition~\ref{th:poly}.]
Assume $\Phi$ to be defined on an open subset $U$ of $T^*M$. Fix a regular point $q_0$ in $\pi(U)$ and let $(\hat{M}, \hat{D})$ be the nilpotent approximation of $(M, D)$ at $q_0$.

Let $\hat \delta$ be a nonzero maximal minor of $\hat A$. It is a $w$-homogeneous polynomial which is the homogeneous part of highest weighted degree of the corresponding minor (same rows and columns) $\delta$ of $A$, which is nonzero as well. It results easily from \eqref{eq:3.8} that, for $k=m+1,\dots, n$, we have  $\Phi_k - \alpha u_k=p_k/\delta$ where $d_w(p_k) \leq d_w(\delta)+2$, and  $\Psi_k= \hat{p}_k/\hat{\delta}$, where $\hat p_k$ is the homogeneous part (eventually zero) of weighted degree $d_w(\delta)+2$ in $p_k$. From the hypothesis of the theorem, $p_k/\delta$ is polynomial, therefore $\Psi_k= \hat{p}_k/\hat{\delta}$ is polynomial as well and by Lemma~\ref{prop:poly} we get $X_i(\alpha)(q_0)=0$, $i=1,\dots,m$.

Since regular points form an open and dense subset of $\pi(U)$, the functions $X_i(\alpha)$, $i=1,\dots,m$, are identically zero on $\pi(U)$. The family $X_1,\dots, X_m$ being a Lie-bracket generating family, we thus obtain that $\alpha$ is locally constant.
\end{proof}

\subsection{A remark on Lemma~\ref{le:APsi=d}}

Let $\hat{\alpha}$ be the real-valued function on $\hat M$ defined by
\begin{equation}
\left\{
\begin{aligned}
		& \hat{\alpha} (0)= \alpha(q_0),\\
		& \hat{X_i}(\hat{\alpha}) \equiv X_i(\alpha)(q_0) \qquad i= 1, \dots, m, \\
    & \hat{X_k}(\hat{\alpha}) \equiv 0,  \qquad k= m+1, \dots, n.
\end{aligned}
\right.
\end{equation}
In a system of privileged coordinates $z$ at $q_0$ such that $X_i(z_j)(q_0)=\delta_{ij}$, $\hat{\alpha}$ writes as
\begin{equation}
\hat{\alpha} = \alpha(q_0) + \sum_{i=1}^m z_i X_i(\alpha)(q_0).
\end{equation}
The existence of the mapping $\Psi$ in Lemma~\ref{le:APsi=d} may be interpreted as follows.
\begin{Lemma}
There exists a fiber-preserving map $\hat{\Phi}: T^* \hat{M} \to T^* \hat{M}$ such that, on a neighbourhood of every ample covector (w.r.t.\ $\hat{g}$), $\hat{\Phi}$ is smooth and sends the integral curves of the Hamiltonian vector fields of the metric $\hat{g}$ to the ones of $\hat{\alpha}^2\hat{g}$.
\end{Lemma}

\begin{proof}
Note that $\hat{d}$ is the vector $d$ constructed by using $\{\hat{X}_1, \dots, \hat{X}_n\}$ as a frame and $\hat{\alpha}$ as conformal coefficient in \eqref{eq:3.11}. Let $\Psi$ be the solution of $\hat{A} \Psi =\hat{d}$ and set
\begin{equation}
\left\{
\begin{aligned}
		& \hat{\Phi}_k = \hat{\alpha}u_k, \qquad k= 1, \dots, m, \\
    & \hat{\Phi}_k = \Psi_k + \hat{\alpha}u_k,  \qquad k= m+1, \dots, n.
\end{aligned}
\right.
\end{equation}
Define $\hat{\Phi}: T^* \hat{M} \to T^* \hat{M}$ as the fiber-preserving map such that $u \circ \hat{\Phi} = (\hat{\Phi}_1, \dots, \hat{\Phi}_n)$.
It results from \cite[Proposition 3.13]{jmz2018} that $\hat{\Phi}$ sends the extremal flows of $\hat{g}$ to the one of $\hat{\alpha}^2\hat{g}$ near any ample covector.
\end{proof}


\section{The most general sub-Riemannian Weyl type theorem}
\label{mainthmsec}

In this section we formulate the most general technical version of the sub-Riemannian Weyl theorem that we were able to obtain. The versions  of the sub-Riemannian Weyl theorem  (Theorems \ref{Weylcor2} and \ref{Weylcor3}) formulated in the Introduction will follow from its proof.

We start with the following.
\begin{Def}
	\label{polydef}
	Given an open subset $U$ of $\C M$ the function $\delta$ on the complexified cotangent bundle $T^*U$ is called a \emph{polynomial with respect to the fibers over $U$}  if, in the canonical coordinates induced by some local coordinates in $U$, $f$ is represented as a polynomial with respect to the fibers with coefficients being holomorphic function of the base $U$. Further, given a point $q_0\in M$ a \emph{germ over  $q_0$ of polynomials with respect to the fibers} of cotangent bundle is an equivalence class of such polynomials so that two polynomials are equivalent if they coincide over a neighborhood of a point $q_0$.
\end{Def}

\begin{Theorem}
	\label{mainthm}	
	Assume that  $(M,D,g)$ is a real analytic sub-Riemannian manifold such that there is no non-constant polynomial $\delta$ with respect to the fibers of $T^*U$ over some open set $U$ of\, $\mathbb CM$, such that  an open subset  of the zero-level set of $\delta$ is a manifold foliated by complex normal extremals each of which projects to non-strictly normal geodesics.
   Then the sub-Riemannian metric $g$ is Weyl projectively rigid. 
\end{Theorem}

	
	
	\begin{proof}
	Consider
	the map $\Psi=\widetilde{\Phi} - \alpha \widetilde{u}$, which is a  solution of \eqref{eq:3.8}. Then by Proposition~\ref{th:poly}, to get the conclusion of our theorem we only need to show that $\Psi$ is polynomial with respect to the fibers near a point $q_0$.
	
	Given a positive integer $k$ denote by
	$A_k$ the truncation up to the $k$th layer of the fundamental matrix
	$A$ from \eqref{eq:3.8}.
	By \cite[Proposition 3.11]{jmz2018} we can choose $k\geq n-m$ large enough so that at least one minor of size $(n-m)\times (n-m)$ in  $A_k$ is  not identically zero.
	
	From now on we work on the complexified manifold $\C M$. 
	The corresponding complexified cotangent bundle $T^*\C M$  can be identified  locally with $\C^n \times \C^n$ (where $q_0$ is identified with $0$). Let $\mathcal{O}_n$ be the set of germs of holomorphic functions on $\C M \simeq \C^n$ at $0$. Under the above identification,
	a germ of polynomials with respect to the fibers of   $T^*\C M$ in the sense of Definition \ref{polydef} can be seen as polynomials on  (the second copy of)  $\C^n$ with coefficients in $\mathcal{O}_n$. Since $\mathcal{O}_n$ is a factorial ring (or unique factorization domain), the set $\mathcal{O}_n[z_1,\dots,z_n]$ of these polynomials form a factorial ring as well (see for instance \cite{Jong2000}), which mean that every element can be written as a product of irreducible elements, uniquely up to order and units.
	
	It results from \eqref{eq:3.8}
	that for any nonzero minor $\delta$ of $A_k$ we have
	\begin{equation}
	\label{Phi'}
	\Psi_i=\cfrac{p_i}{\delta},\quad m+1\leq i\leq n,
	\end{equation}
	where $\delta$ and $p_i$ are polynomial in $\mathcal{O}_n[z_1,\dots,z_n]$. Canceling the greatest common factor of the collection of polynomials $\{\delta,p_{m+1},\ldots, p_n\}$, we get a collection of polynomials  $\{\widetilde\delta,\widetilde p_{m+1},\ldots, \widetilde p_n\}$ with the greatest common factor equal to constant and such that
	\begin{equation}
	\label{Phi1'}
	\Psi_i=\cfrac{\widetilde p_i}{\widetilde\delta}.
	\end{equation}
	Besides, substituting $\eqref{Phi1'}$ into \eqref{eq:3.7} we get
	\begin{equation}
	\label{divis}
	\vec{h}(\widetilde\delta) \widetilde p_i \text { is divisible by } \widetilde \delta.
	\end{equation}
	
	Let us show that under the assumption of Theorem~\ref{mainthm} $\widetilde \delta$ is constant. Assuming the converse, there is an irreducible polynomial $\delta_1$ in $\mathcal{O}_n[z_1,\dots,z_n]$ such that $\widetilde\delta =  \delta_1^s p$, where $s$ is a positive integer and $p$ is a polynomial such that $p$ and $\delta_1$ are coprime. By constructions, there exists $j\in \{m+1,\ldots n\}$
	such that $\widetilde p_j$ is not divisible by $\delta_1$, otherwise $\delta_1$ is a nonconstant common factor of the collection $\{\widetilde\delta,\widetilde p_{m+1},\ldots, \widetilde p_n\}$.
	
	Consider this particular $j$. Although  the polynomials $\widetilde p_j$ and $\widetilde\delta$ are not coprime in general, if we further reduce the expression \eqref{Phi1'} for $\Psi_j$ to the lowest terms (i.e.\ such that the numerator and denominator will be coprime), then the denominator will be divisible by $\delta_1$. Note that in  \eqref{Phi'} we can use any nonzero maximal minor $\delta$ of $A_k$ and  the expression for $\Psi_j$ in the lowest terms is unique and does not depend on the initial choice of the nonzero maximal minor $\delta$. Hence, $\delta_1$ is a common divisor of all maximal minors of $A_k$.
	
	From \eqref{divis}, there holds
	$$
	\vec{h}(\widetilde\delta) \widetilde p_j= \Bigl(s p \delta_1^{s-1} \vec{h}(\delta_1) + \delta_1^s \vec{h}(p) \Bigr) \widetilde p_j \text { is divisible by } \widetilde\delta=  \delta_1^s p,
	$$
	which implies that $\Bigl(s p  \vec{h}(\delta_1) + \delta_1 \vec{h}(p)\Bigr) \widetilde p_j$ is divisible by $\delta_1 p$, and so that
	$$
	p \widetilde p_j \vec{h}(\delta_1)  \text { is divisible by } \delta_1 .
	$$
	Since $\widetilde p_j$ and $p$ are not divisible by the irreducible polynomial $\delta_1$, we conclude that $\vec{h}(\delta_1)$ is divisible by $\delta_1$.
	
	Denote by $\mathcal S_1$ the zero-level set of $\delta_1$ (it is an analytic subset of $T^*\C M \simeq \C^n \times \C^n$ near the fiber $\{0\} \times \C^n$). We have shown that $\vec{h}(\delta_1)$ is zero on ${\mathcal S}_1$. Note that, although $\delta_1$ is irreducible over $\mathcal O_n$, its restriction to some fiber might be reducible over $\mathbb C$. However, the restriction on the generic fiber  (in the domain of definition of $\delta_1$) is irreducible.  This implies that $d\delta_1$ is not identically zero on $\mathcal{S}_1$.
	Indeed, assume the converse.  Since $\delta_1$ is not constant, there exists $k$ such that ${\partial\over\partial u_k}\delta_1$ is not zero but  ${\partial\over\partial u_k}\delta_1=0$ on $\mathcal{S}_1$. By applying the Hilbert Nullstellensatz to each generic fiber ${\partial\over\partial u_k}\delta_1$  must belong to the radical of the ideal generated by the restriction of $\delta_1$ to the same fiber, but this is impossible as degree of the polynomial ${\partial\over\partial u_j}\delta_1$ is smaller than degree  of  $\delta_1$.
	
	Denote by $\mathcal{S}_2$  the subset of $\mathcal{S}_1$, where the vector $\left({\partial \delta_1\over\partial u_1}, \ldots, {\partial \delta_1\over\partial u_n}\right)$\footnote{Here we could impose a weaker condition $d\delta_1\neq 0$ but we will need the given stronger condition  in the next section} is not equal to zero. By constructions this is an  open (and dense) subset of $\mathcal{S}_1$. Then, by constructions,  $\mathcal{S}_2$ is a submanifold of $\mathcal{S}_1$ and  $\vec{h}$ is tangent to ${\mathcal S}_2$. Therefore, any complexified normal extremal of $(M, D,g)$ starting at a point of ${\mathcal S}_2$ will stay in  ${\mathcal S}_2$ for sufficiently small time and $\mathcal{S}_2$ is foliated by normal extremals.

Consider such a normal extremal $\lambda(t)$, $t \in [0,T]$, in ${\mathcal S}_2$, so that $\delta_1(\lambda(t))\equiv 0$.  Given any $\lambda\in T^*\mathbb CM$  define the filtration $\{ J_{\lambda}^{(k)}\}_{k\in \mathbb N}$ as follows. First, set  $J_{\lambda}^{(0)}$ to be the tangent space  to  the fiber of $T^*\mathbb CM$ at $\lambda$, i.e.
$$J_{\lambda}^{(0)} =  \Big\{ v \in T_{\lambda}(T^* M) \ : \ \pi_* v = 0 \Big\},
$$
where, as before, $\pi: T^*\mathbb CM\mapsto \mathbb M$ denotes the canonical projection. Note that $J^{(0)}$ defines a distribution on $T^*\mathbb CM$ called the \emph{vertical distribution}.  Finally, define recursively
\begin{equation*}
J_{\lambda}^{(k)} =  J^{(k-1)}_{\lambda} + \mathrm{span}\Bigl\{(\mathrm{ad}\,\vec{h})^{k} V(\lambda): V \mathrm{\,\, is \,\,a \,\,local \,\,section \,\,of\,\,} J^{(0)}\Bigr\}.
\end{equation*}
Since $\delta_1$ is a common divisor of all maximal minors of $A_k$, due to \cite[Lemma 3.12]{jmz2018} there holds $\dim J_{\lambda(t)}^{(k+1)} < 2n$ for all $t \in [0,T]$, and so
\begin{equation}
\label{Jn-m}
\dim J_{\lambda(t)}^{(n-m)} < 2n \quad \forall t \in [0,T]
\end{equation}
(recall that from the beginning we have taken $k\geq n-m$). Further, for $t$ in an open and dense subset of $[0,T]$  the dimensions of the spaces $J_{\lambda(\tau)}^{(n-m)}$ are constant for every $\tau$ sufficiently close  to $t$. This and \eqref{Jn-m} imply that
for these times $t$ we have, for any $s \in \N$, $\dim J_{\lambda(t)}^{(s)}\leq \dim J_{\lambda(t)}^{(n-m)} < 2n$.
As a consequence, $\gamma(\cdot))=\pi (\lambda(\cdot))$ is a non-strictly normal geodesic of $(M, D,g)$ (it results from \cite[Prop.\ 3.12]{ref2} since our sub-Riemannian manifold is analytic). This completes the proof of our theorem as we can take $\mathcal{S}_2$ as the required open set in the zero-level set of $\delta$ and $U=\pi({\mathcal S}_2)$.
\end{proof}

\section{ Proof of Theorems \ref{Weylcor2} and \ref{Weylcor3}}
\label{famsec}
First we prove Theorem \ref{Weylcor3}.
	Assume by contradiction that the metric $g$ is not Weyl projectively rigid. Then by Proposition \ref{th:poly}
$\Psi$ is not a polynomial with respect to the fibers and we can repeat the arguments of proof of Theorem \ref{mainthm} in the previous section to find a polynomial $\delta_1$ such that the open  set $\mathcal{S}_2$ of the zero-level set of $\delta_1$ is foliated by non-strictly normal complex geodesics. We consider two cases separately.
	
	{\bf Case 1} \emph{$\mathcal{S}_2$ does not belong to any  level set of $h$}. Then there exists a nonzero $c\in\mathbb C$ such that  $\mathcal{S}_2$ is transversal to the $c$-level set of $h$ at some point $\lambda_0$. By construction  of $\mathcal{S}_2$ the vector $\left({\partial \delta_1\over\partial u_1}, \ldots, {\partial \delta_1\over\partial u_n}\right)$ is not zero.  Then
		for every point $q\in M$ sufficiently close to $\pi (\lambda_0)$, the intersection of an open subset of ${\mathcal S}_2$ (that for simplicity of notation will be called  $\mathcal{S}_2$ in the sequel) with the  fiber $T_q^*\mathbb C M$ and the $c$-levels set of $h$  is of dimension $n-2$ .
	
	It gives rise to a $(n-2)$-parameters family of complexified normal extremals which project to non-strictly normal geodesics passing through $q$. The corank of these geodesics may vary, but there is an integer $\kappa$,  an $\varepsilon>0$,  and  an open subset $V$ of $\mathcal S_2\cap T_q\mathbb C M\cap h^{-1}(c)$ such that for every $\lambda\in V$ the extremal $\{e^{t\vec h}\lambda, |t|<\varepsilon\}$ projects to a geodesic of corank $\kappa$. 	
		Thus we proved that there exists, for at least one positive integer $\kappa$, an $(n-2-\kappa)$-parametric family of corank $\kappa$ non-strictly normal complexified geodesics through a point, which contradicts our assumption. This concludes the proof of Theorem \ref{Weylcor3} in the considered case.
		
{\bf Case 2} \emph{$\mathcal{S}_2$ belongs to a $c$-level set of $h$ for some $c\in \mathbb C$}. Note that  $\mathcal S_2$  and $h^{-1}(c)$ have the same dimension, so $\mathcal S_2$ is an open subset of $h^{-1}(c)$.
	Then
	for every point $q\in \pi (\mathcal{S}_2)$  the intersection of an open subset of ${\mathcal S}_2$ (that for simplicity of notation will be called  $S_2$ in the sequel) with the  fiber $T_q^*\mathbb C M$   is of dimension $n-1$. Then we can arrive to a contradiction in the same way as in the last paragraph of the previous case.

%
%
		
	Now let us prove Theorem \ref{Weylcor2}.
	 Let  $(\hat{M}, \hat{D}, \hat{g})$ be the nilpotent approximation of $(M, D,g)$ at $q_0$, and $\Psi$ be the map  given by Lemma~\ref{le:APsi=d}. If we will show that $\Psi$ is polynomial, then the conclusion of the theorem will then follow from Lemma~\ref{prop:poly} and the same argument as in the end of the proof of Proposition~\ref{th:poly}.
	 To prove that $\Psi$ is polynomial we can literally repeat the arguments of the proof of Theorem \ref{mainthm} and of the previous two paragraphs for the nilpotent approximation  $(\hat{M}, \hat{D}, \hat{g})$.

\section{ Proof of Weyl's type theorems in terms of abnormal extremals}
\label {gensec}

 \subsection{Abnormal extremals and their properties} First,  describe  the complex  abnormal extremals of a distribution $D$ of rank $m$. The construction is the complexification of the standard geometric construction of usual (i.e. real) abnormal extremals from Optimal Control Theory (for details, see for example \cite{chitour} and \cite{jacsymb}).

Let  $\sigma$ be the $2$-form on $T^*\mathbb C M$ which is the complexification of the standard symplectic structure on $T^*M$.
Also denote by $(T\mathbb C M)^\perp$ the zero section of $T^*\mathbb C M$.

\begin{Def}
	\label{abndef}
	An unparametrized Lipschitzian complex curve in $\mathbb  C D^\perp\backslash (T \mathbb C M)^\perp$ is called
	\emph{an abnormal extremal} of a distribution $\mathbb C D$ if the (complex) tangent line to it at
	almost every point belongs to the kernel of the restriction
	$\sigma|_{\mathbb C D^\perp}$ of $\sigma$ to $\mathbb C D^\perp$ at this point.
\end{Def}

Directly from this definition it follows that any abnormal extremal belongs to the set
\begin{equation}
\label{tildeWD}
\widetilde W_D:=\{\lambda\in \mathbb C D^\perp:{\rm ker}\bigl(\sigma|_{\mathbb C D^\perp}(\lambda)\bigr)\neq 0\}.
\end{equation}

Since ${\rm codim} \, D^\perp =\mathrm{rank} \,D=m$, from the elementary properties of skew-symmetric matrices and forms it follows that
\begin{equation}
\label{tildeWDdescription}
\widetilde W_D=\begin{cases}\C D^\perp& \text{if } m\text { is odd,}\\ \{\lambda\in \C D^\perp:\Bigl (\bigwedge ^{n-m/2}(\sigma|_{\C D^\perp})\Bigr)(\lambda)=0\} &  \text{if } m\text { is even,}
\end{cases}
\end{equation}
where $\bigwedge ^{s}(\sigma|_{\C D^\perp})$ denotes the $s$th wedge-power of the form $\sigma|_{\C D^\perp}$. Note that for every odd rank distributions and for generic even rank distributions
\begin{equation}
\label{tildeWDdim}
\dim \widetilde W_D=\begin{cases}2n-m& \text{if } m\text { is odd,}\\ 2n-m-1 &  \text{if } m\text { is even.}
\end{cases}
\end{equation}
Note that in general for even rank distribution $D$ the set $\widetilde W_D$ is not smooth at every point but it is smooth for generic distributions of this class.

Now define  the following subset $W_D$ of $\widetilde W_D$ as follows:
\begin{equation}
\label{WD}
W_D:=\{\lambda\in \widetilde W_D:{\rm ker}\bigl(\sigma|_{\widetilde W_D}(\lambda)\bigr) \text{ is one-dimensional}\}.
\end{equation}
(for the detailed description of the set $W_D$ in terms of a local basis of distribution see \cite{jacsymb}).

By constructions, the kernels of $\sigma|_{W_D}$ form the \emph{characteristic complex rank $1$ distribution} ${\mathcal C}$ on $W_D$ and
the complex integral curves  of this distribution are
complex abnormal extremals of the distribution $D$. Moreover, these are all complex abnormal extremals which lie entirely in $W_D$. Following \cite{chitour},
we give the following definition.
\begin{Def}
	\label{minorddef}
We say that a complex abnormal extremal $\Gamma$ and the corresponding abnormal geodesics are  of \emph{minimal order} if the set $\Gamma\cap W_D$ has a full measure in $\Gamma$.
\end{Def}

\begin{rk}
	\label{open}
Since $W_D$ is open, the set $\Gamma\cap W_D$ is open in $\Gamma$.
\end{rk}

\subsection{The key lemma making the passage from normal to abnormal extremals}
\label{abnrem2sec} First we prove the following.

\begin{Lemma}
\label{abnrem3}
Let $D$ be a real analytic distribution such that every complex abnormal extremal of $D$  is of minimal order.
Assume that  for a sub-Riemannian metric $g$ on $D$ there exists
a nonzero polynomial $\delta$ with respect to the fibers of $T^*U$ over some open set $U$ of $\mathbb CM$, such that
 an open subset  of the zero-level set of $\delta$ is
foliated by complex normal extremals of $g$  which are lifts of non-strictly normal geodesics. Then there exists a submanifold $\mathcal S_0$  of codimension
one in the projectivization $\mathbb PT^*\mathbb C M$ of $T^*\mathbb C M$ foliated by complex abnormal extremals of $D$. Moreover, we can choose $\mathcal S_0$ such that it belongs to the projectivization $\mathbb P W_D$ of $W_D$.
\end{Lemma}

\begin{proof}
Let $\mathcal{S}$ be a smooth part of the zero level set of $\delta$. Given $\lambda\in \mathcal S$ take the geodesic $\gamma_\lambda$ obtained by the projection to $M$ of the sub-Riemannian extremal passing through $\lambda$ and let $\mathcal A_\lambda$ be the affine space of all normal lifts of $\gamma_\lambda$. Note that  $\mathcal A_\lambda$ can be seen as an affine subspace of $T_{\pi(\lambda)}^*\mathbb CM$, as any lift is uniquely defined by its intersection with $T_{\pi(\lambda)}^*\mathbb C M$. Also, clearly $\lambda\in \mathcal{A}(\lambda)$, and by the non-strictly normal assumption, $\dim \,\mathcal A_\lambda>0$. This dimension may vary, but there is an integer $\kappa>0$  and an open subset $V$ of $\mathcal S$ such that  $\dim \,\mathcal A_\lambda=\kappa$ for every $\lambda\in V$. Further, let $\mathcal B_\lambda=\mathcal S\cap \mathcal A_\lambda$. Let $\mathcal E_\lambda$ be the vector space corresponding to the affine space $\mathcal A_\lambda$ and $\mathbb P \mathcal E_\lambda$ its projectivization.  In the sequel, given an element $e\neq 0$ of  a vector space,  $[e]$ will denote the line generated by $e$, i.e.\ the element of the corresponding projective space. Note that the line  $[\lambda]$ in $T_{\pi(\lambda)}^*\mathbb M$  does not
	belong to $\mathcal B_\lambda$, because two  extremals passing through two different points of the line $[\lambda]$ project to the same unparametrized curve, but the parametrizations on this curve, induced by these two projections, are different (one is a non-unit scalar multiple of the other one).

Let us distinguish the following pair of alternative properties:
\begin{enumerate}
	\item  $\mathcal B_\lambda$ is  a $(\kappa-1)$-dimensional manifold;
	\item $\mathcal B_\lambda$ is a $\kappa$-dimensional manifold;
\end{enumerate}
It is easy to see that, maybe after shrinking $V$, either property (1) or property (2)  hold simultaneously for every $\lambda\in V$.

	Take some codimension $1$ submanifold $W$ of $V$.  If $\kappa>1$ we  take $W$ being transversal to  $\mathcal{B}_\lambda $ for every $\lambda$. For  every $\lambda \in W$ choose $f(\lambda)\in \mathcal A_\lambda$ smoothly in $\lambda$ and such that the restriction of the map $\lambda\mapsto  [f(\lambda)-\lambda]$ to the set $\mathcal B_\lambda\cap W$ is injective. The latter is possible because
\begin{equation}
\label{dimcompar}
\dim \,(\mathcal B_\lambda\cap W)\leq \dim \mathbb P \mathcal E_\lambda.
\end{equation}
Indeed, the right hand-side of~\eqref{dimcompar} is equal to $\kappa-1$, whereas the left hand-side of \eqref{dimcompar} is equal to $\kappa-2$ if condition (1) holds with $\kappa>1$; or to $\kappa-1$ if either condition (2)  or condition (1) with $\kappa=1$ holds.

By Remark  \ref{abnrem1} there is a projectivized  abnormal lift $\Gamma_\lambda$ of $\gamma_{\lambda}$ passing through $[f(\lambda)-\lambda]$. Let
\begin{equation}
\mathcal S_0=\bigcup_{\lambda \in W} \Gamma_\lambda.
\end{equation}	
The imposed injectivity conditions ensures that to different normal extremals, foliating $\mathcal{S}$,
different abnormal extremals are assigned. However, in general the map $\lambda\mapsto [f(\lambda)-\lambda]$ on the whole $W$ is not injective so that the family $\{\Gamma_\lambda\}_{\lambda\in W}$ of curves in general do not foliate $\mathcal S_0$, because some of them intersect. However, now we can use the assumption that all abnormal extremals are of minimal order: fix $\lambda_0\in W$, then there exists $\lambda\in \Gamma_{\lambda_0}\cap \mathbb P W_D$, where $\mathbb P W_D$ is the projectivization of $W_D$.
By Remark \ref{open} there exists a connected neighborhood $V_0$ of $\lambda$ in $\mathcal S_0\cap \mathbb P W_D$ and this neighborhood is foliated by (connected components) of $\Gamma_\lambda \cap V_0$. This is the codimension one submanifold of  $\mathbb PT^*\mathbb C M$  desired in the lemma.

\end{proof}

%

	


\subsection{ Proof  of  Theorem \ref{Weylcor3a}}
Assume by contradiction that the metric $g$ is not Weyl projectively rigid. Then by Theorem~\ref{mainthm} there exists a polynomial $\delta$ such that the open level set $\mathcal{S}$ of its zero-level set is foliated by non-strictly normal complex geodesics. Then we can choose a codimension 1  submanifold $\mathcal S_0$ of $\mathbb P T^* \C M$ as in Lemma \ref{abnrem3}. This means that $\dim \mathcal S_0=2n-2$. On the other hand $\mathcal S_0$ is a submanifold of $\mathbb PW_D$ and by \eqref{tildeWDdim} for $m>1$ we have
$$2n-2=\dim \mathcal S_0\leq \dim \mathbb P \widetilde W_D\leq 2n-3,$$ which leads to the contradiction.

\subsection{ Proof  of  Theorem \ref{Weylcor2a}}
		
Now let us prove Theorem \ref{Weylcor2}.
Let  $(\hat{M}, \hat{D}, \hat{g})$ be the nilpotent approximation of $(M, D,g)$ at $q_0$, and $\Psi$ be the map  given by Theorem~\ref{mainthm}. If we will show that $\Psi$ is polynomial, then the conclusion of the theorem will then follow from Lemma~\ref{prop:poly} and the same argument as in the end of the proof of Proposition~\ref{th:poly}.
To prove that $\Psi$ is polynomial we can literally repeat the arguments of the proof of Theorem \ref{mainthm} and of the previous subsection for the nilpotent approximation  $(\hat{M}, \hat{D}, \hat{g})$.

\subsection {A refinement of the condition of minimal order}
\label{weaksec}
Assume that there exists a nested set collection $\{\widetilde W_D^{(i)}\}_{i=0}^N$ in $\mathbb P D^\perp$ such that the following conditions hold:
\begin{enumerate}
	\item $\widetilde W_D^{(0)}=\widetilde W_D$;
	\item $\widetilde W_D^{(i+1)}\subsetneq \widetilde W_D^{(i)}$ for every  $i\in\{0,\ldots N-1$\};
	\item The set $W_D^{(i)}:=\widetilde W_D^{(i)}\backslash \widetilde W_D^{(i+1)}$ is a manifold  and  $d_i:=\dim\, W_D^{(i)}$ is strictly decreasing with respect to $i$, $i\in\{0,\ldots N-1\}$;
	\item
	\begin{itemize}
		\item
 If $d_i$ is odd, then the kernel of restriction of the canonical symplectic form $\sigma$ to $W_D^{(i)}$ is
 one-dimensional at every point of $W_D^{(i)}$;
 \item If $d_i$ is even, then the subset $\bar W_D^{(i)}$ of $W_D^{(i)}$,  consisting of all points for which the restriction of $\sigma$ to $W_D^{(i)}$ has nonzero kernel, is  a codimension one submanifold of $W_D^{(i)}$ such the kernel of the restriction of $\sigma$ to $\bar W_D^{(i)}$ is one-dimensional at every point of $\bar W_D^{(i)}$
 \end{itemize}
	\end{enumerate}

\begin{Def}
	\label {weakmindef}
We say that a complex abnormal extremal $\Gamma$ and the corresponding abnormal geodesics have \emph{weakly minimal order} if there exist $i\in\{0,\ldots, N\}$ such that the set $\Gamma\cap W_D^{(i)}$ has a full measure in $\Gamma$.
\end{Def}

Obviously a complex abnormal extremal of minimal order has also weakly minimal order (it just corresponds to $i=0$ in Definition \ref{weakmindef})

From the proof of Lemma \ref{abnrem3} it follows immediately that \emph{we can replace the condition that all abnormal extremals of $D$ are of minimal order by the condition that all abnormal extremals of $D$ have weakly minimal order in the formulation of Lemma \ref{abnrem3} and consequently in  the formulations of Theorems \ref{Weylcor2a} and  \ref{Weylcor3a} and Corollary \ref{Cor:anyrankrigid}}.

\end{document}